\documentclass[12pt]{amsart}
\usepackage{txfonts}
\usepackage{subfig}
\usepackage{amscd,amssymb,mathabx}
\usepackage[arrow,matrix,graph,frame,poly,arc,tips]{xy}
\usepackage{graphicx,calrsfs}
\usepackage{mathtools}
\usepackage{tikz}
\usetikzlibrary{decorations.markings}
\usetikzlibrary{shapes}
\usetikzlibrary{arrows}
\usetikzlibrary{backgrounds}
\usetikzlibrary{calc}
\usepackage{mdwlist}
\usepackage{multirow}
\usepackage{enumerate}
\usepackage{verbatim}
\usepackage{etoolbox}
\AtEndEnvironment{proof}{\setcounter{claim}{0}}

\DeclarePairedDelimiter{\form}{\langle}{\rangle}

\DeclarePairedDelimiter{\floor}{\lfloor}{\rfloor}
\DeclarePairedDelimiter{\brac}{[}{]}
\DeclarePairedDelimiter{\abs}{\lvert}{\rvert}

\newcommand\ba{\begin{align*}}
\newcommand\ea{\end{align*}}
\newcommand\be{\begin{enumerate}}
\newcommand\ee{\end{enumerate}}
\newcommand\bp{\begin{proof}}
\newcommand\ep{\end{proof}}
\newcommand\bpp{\begin{prop}}
\newcommand\epp{\end{prop}}
\newcommand\bpb{\begin{prob}}
\newcommand\epb{\end{prob}}
\newcommand\bd{\begin{defn}}
\newcommand\ed{\end{defn}}
\newcommand\bh{\begin{hint}}
\newcommand\eh{\end{hint}}

\newcommand\bN{\mathbb{N}}
\newcommand\N{\mathbb{N}}
\newcommand\bR{\mathbb{R}}

\newcommand\bQ{\mathbb{Q}}

\newcommand\bZ{\mathbb{Z}}

\newcommand\Hom{\operatorname{Hom}}
\newcommand\SO{\operatorname{SO}}

\newcommand\Id{\operatorname{Id}}

\DeclareMathOperator\Homeo{Homeo}

\newcommand\sse{\subseteq}
\newcommand\co{\colon}

\DeclareMathOperator\Diff{Diff}

\newcommand\rot{\operatorname{rot}}

\renewcommand{\MR}[1]
{\href{http://www.ams.org/mathscinet-getitem?mr=#1}{MR#1}}

\def\thetitle{{Integrability of moduli and regularity of Denjoy counterexamples}}
\def\theauthors{{Sang-hyun Kim and Thomas Koberda}}
\usepackage{hyperref}
\hypersetup{
   colorlinks=false,
   plainpages,
   urlcolor=black,
   linkcolor=black
   pdftitle=   \thetitle,
   pdfauthor=  {\theauthors}
}

\theoremstyle{theorem}
\newtheorem{thm}{Theorem}[section]
\newtheorem{lem}[thm]{Lemma}
\newtheorem{cor}[thm]{Corollary}
\newtheorem{prop}[thm]{Proposition}

\newtheorem{que}[thm]{Question}
\newtheorem*{claim*}{Claim}

\theoremstyle{remark}
\newtheorem{exmp}[thm]{Example}
\newtheorem{rem}[thm]{Remark}

\theoremstyle{definition}
\newtheorem{defn}[thm]{Definition}
\newtheorem{prob}{Problem}[section]

\begin{document}
\title\thetitle
\date{\today}

\author[S. Kim]{Sang-hyun Kim}
\address{School of Mathematics, Korea Institute for Advanced Study (KIAS), Seoul, 02455, Korea}
\email{skim.math@gmail.com}
\urladdr{http://cayley.kr}

\author[T. Koberda]{Thomas Koberda}
\address{Department of Mathematics, University of Virginia, Charlottesville, VA 22904-4137, USA}
\email{thomas.koberda@gmail.com}
\urladdr{http://faculty.virginia.edu/Koberda}

\begin{abstract}
We  study the regularity of exceptional actions of groups by $C^{1,\alpha}$ diffeomorphisms on the circle, i.e. ones
which admit exceptional minimal sets, and whose elements have first derivatives that are continuous with
concave modulus of continuity $\alpha$.
Let $G$ be a finitely generated group admitting a $C^{1,\alpha}$
action $\rho$ with a free orbit on the circle, and such that the logarithms of derivatives of group elements
are uniformly bounded at some point of the circle.
We prove that if $G$ has spherical growth bounded by $c n^{d-1}$ and if the function $1/\alpha^d$ is integrable near zero,
then under some mild technical assumptions on $\alpha$, there is a sequence of exceptional $C^{1,\alpha}$ actions of $G$
which converge to $\rho$ in the $C^1$ topology. As a consequence for a single diffeomorphism, we obtain that
if the function $1/\alpha$ is integrable near zero, then there exists a $C^{1,\alpha}$ exceptional
diffeomorphism of the circle. This corollary accounts for all previously known moduli of continuity for derivatives of
 exceptional diffeomorphisms. 
 We also obtain a partial converse to our main result.
 For finitely generated free abelian groups, the existence of an exceptional
 action, together with some natural hypotheses on the derivatives of group elements, puts integrability restrictions on the modulus
 $\alpha$. These results are 
related to a long-standing question of D. McDuff concerning the length spectrum of exceptional $C^1$ diffeomorphisms of the circle.
\end{abstract}

\maketitle

\section{Introduction}

Let $f\in\Homeo^+(S^1)$ be an orientation preserving homeomorphism of the circle without any periodic points. It is well-known that in this
case $f$ has an irrational rotation number $\theta$. Here, we make the identification $S^1=\bR/\bZ$, and we have the rotation number
\[\mathrm{rot}(f)=\lim_{n\to\infty}\frac{F^n(x)-x}{n}\pmod 1\in\bR/\bZ,\] where $F$ is any lift of $f$ to $\bR$ and $x\in\bR$ is arbitrary. It is well-known
that $\mathrm{rot}(f)$ is independent of $x$ and of the choice of a lift.

In this introduction, we shall always assume that $f\in\Homeo^+(S^1)$ has irrational rotation number $\theta$
unless otherwise noted. A standard
fact going back to Poincar\'e
asserts that if $f$ has a dense orbit, then $f$ is topologically conjugate to an irrational rotation by $\theta$.
If $f$ does not have a dense orbit then
it must have a \emph{wandering interval}, which is to say a nonempty interval $J$ such that $f^n(J)\cap J=\varnothing$ for $n\neq 0$.

A classical result of Denjoy asserts that if $f$ is twice differentiable (or if in fact the logarithm of the
derivative of $f$ has bounded variation, equivalently just the derivative of $f$ has bounded variation),
then $f$ is topologically conjugate to a rotation by $\theta$. In lower levels of regularity, this fact
ceases to hold. It is easy to produce continuous examples which are not topologically conjugate to a rotation, and Denjoy showed that one
can construct differentiable examples for every irrational rotation number $\theta$.
Examples of this type were also  know to Bohl~\cite{Bohl1916}.
Such examples will be called \emph{exceptional} diffeomorphisms, since they have a so-called
exceptional minimal set, which in this case will be homeomorphic to a Cantor set (see Theorem 2.1.1 of~\cite{Navas2011}).
In general, an \emph{exceptional group action} is one which admits an exceptional minimal set.
In what follows, we will consider not just single diffeomorphisms, but also non-cyclic groups acting with exceptional minimal sets on the circle,
and the regularity which can be required of such actions. 

\subsection{Main results}
In this paper, we consider the problem of determining which moduli of continuity can be imposed on the derivatives of elements
in exceptional diffeomorphism groups of the circle. 
 
A \emph{(concave) modulus of continuity} (or, \emph{concave modulus}) is a homeomorphism
\[\alpha\colon[0,\infty)\to[0,\infty)\]
that is concave as a map. 
We say a function $g\colon S^1\to\bR$ is \emph{$\alpha$--continuous} if it satisfies \[\sup_{x\ne y} \frac{|g(x)-g(y)|}{\alpha(|x-y|)}<\infty.\] Here, we interpret $|x-y|$ by identifying $S^1$ with $\bR/\bZ$
and computing this difference modulo $1$.
The value of this supremum
is sometimes called the $\alpha$--norm of $g$, and is written $[g]_{\alpha}$. 
Conversely, every continuous map $g$ on $S^1$ is $\alpha$--continuous for some concave modulus $\alpha$; see~\cite{CKK2019} for instance.

A $C^1$--diffeomorphism $f\co S^1\to S^1$ is said to be $C^{1,\alpha}$ if $f'$ is $\alpha$--continuous. One commonly denotes by $\Diff_+^{1,\alpha}(S^1)$ the set of orientation--preserving $C^{1,\alpha}$ diffeomorphisms of $S^1$.
It follows from the concavity of $\alpha$ that $\Diff_+^{1,\alpha}(S^1)$ is a group.

Note that if $\alpha(x)=x$, then $\alpha$--continuity is just Lipschitz continuity. More generally, if $\alpha(x)=x^{\tau}$ for some $0<\tau\leq 1$
then $\alpha$--continuity is $\tau$--H\"older continuity. For every $0<\tau< 1$, it is known that there exist $C^{1,\tau}$ exceptional
diffeomorphisms of the circle for arbitrary $\theta$ (see~\cite{Herman1979}, also Theorem 3.1.2 of~\cite{Navas2011}). In fact,
Herman~\cite{Herman1979} proved that for arbitrary $\theta$ there exist $C^{1,\alpha}$ exceptional diffeomorphisms of the circle for
\[\alpha(x)=x(\log{1/x})^{1+\epsilon}\] for all $\epsilon>0$, which implies the corresponding conclusion for H\"older moduli. Moreover,
he proved that such diffeomorphisms can be chosen arbitrarily $C^1$--close to a rotation by $\theta$, and with
uniformly bounded $C^{1,\alpha}$--norms.
On the other hand, he showed that  if $\theta$ satisfies a certain strong Diophantine hypothesis then for
\[\alpha(x)=x\left(\log\log\log{1/x}\right)^{1-\epsilon},\]
every $f\in\Diff_+^{1,\alpha}(S^1)\cap\rot^{-1}(\theta)$ is topologically conjugate to a rotation~\cite[X.4.4]{Herman1979}.
It is a well-known open problem to determine whether or not there exist $C^{1,\alpha}$ exceptional diffeomorphisms
for \[\alpha(x)=x\log{1/x}.\]

Before stating our results, we introduce some notation and terminology.
For a single diffeomorphism $f$, we consider a maximal wandering set of open intervals for $f$ in the circle and write them
as $\{J_i\}_{i\in\bZ}$. These intervals are characterized by
 the property that $f(J_i)=J_{i+1}$, and that any interval properly containing $J_i$ for any index
$i$ must meet the exceptional minimal set of $f$. We will write $\ell_i$ for the length of the interval $J_i$, and we will
oftentimes refer to the collection
$\{J_i\}_{i\in\bZ}$ as a \emph{maximal wandering set}. Observe that a maximal wandering set may not be unique. An exceptional
diffeomorphism of the circle will necessarily have a wandering interval.
The discussion in this paragraph also generalizes to arbitrary group actions.
An exceptional action $\rho$ necessarily admits a wandering interval,
which in this case is an interval $J$ such that every pair of intervals in $\{g J\}_{g\in G}$ are either equal or disjoint. 

\begin{thm}\label{t:c1al-intro}
Let $d\in\bN$, and let $\alpha$ be a concave modulus satisfying 
\[
\int_0^1 1/\alpha(x)^ddx<\infty.\]
If $d>1$, we further assume \[\sup_{0<y<1}  \alpha\!\left( y^{d+1}/\alpha(y)^d\right)/y<\infty.\]
Suppose that a  finitely generated group $G$ admits an action
$\rho\co G\to \Diff_+^{1,\alpha}(S^1)$ with a free orbit $\rho(G).0$ such that  \[\sup_{g\in G} \abs*{\log g'(0)}<\infty.\]
Suppose furthermore that the spherical growth function of $G$ is at most $c n^{d-1}$ for some $c>0$.
Then $\rho$ admits a sequence of exceptional actions
\[\{\rho_k\co G\to\Diff_+^{1,\alpha}(S^1)\}\]
such that $\rho_k\to \rho$ in the $C^1$--topology. Moreover, we may require that
\[\sup_k\{[\rho_k(g)]_{\alpha}\}<\infty\] for all $g\in G$.
\end{thm}

In the statement of Theorem~\ref{t:c1al-intro}, the spherical growth function
of a finitely generated group denotes the number of elements of  length exactly $n$ with respect to some
 generating set. We remark that finite $G$--orbits are allowed under the hypotheses of the theorem.
 
 For a single diffeomorphism (i.e. $G=\bZ$), we have the following immediate consequence  of Theorem~\ref{t:c1al-intro}:

\begin{thm}\label{t:int-alpha}
Let $\theta$ be arbitrary, and suppose that we have \[\int_0^1\frac{1}{\alpha(x)}\,dx<\infty.\]
Then there exists a $C^{1,\alpha}$ exceptional diffeomorphism $f$ of $S^1$ with rotation number $\theta$. Moreover, we may arrange for $f$ to be arbitrarily $C^{1}$--close to a rotation by $\theta$.
\end{thm}

Theorem~\ref{t:int-alpha} recovers all previously known possible moduli of continuity for derivatives of
 exceptional diffeomorphisms of the circle.
In particular, Theorem~\ref{t:int-alpha} recovers the fact
that there are $C^{1,\tau}$ exceptional diffeomorphisms for every irrational rotation number for
all $0<\tau<1$, as well as Herman's corresponding result for $\alpha(x)=x(\log{1/x})^{1+\epsilon}$. Moreover, we can immediately assert the existence
of exceptional $C^{1,\alpha}$ diffeomorphisms for moduli which were previously unrecorded in the literature, such as
\[\alpha(x)=x(\log{1/x})\left(\log\log{1/x}\right)^{1+\epsilon}.\]

The next result is an attempt to control the integrability properties of the modulus, given the existence of an exceptional action. In the
statement, a semi-conjugacy is a monotone, surjective, degree one map which
intertwines two  actions of a group; see Section~\ref{sec:blowup}
below.

\begin{thm}\label{t:al-inv-intro}
Let $\alpha$ be a concave modulus, and let 
\[\rho\co \bZ^d\to \Diff_+^{1,\alpha}(S^1)\]
be an exceptional action. 
Suppose that $\rho$ is semi-conjugate to a faithful action 
\[\bar\rho\co\bZ^d\to\Diff_+^1(S^1)\]
by a semi-conjugacy map  $H$ such that $\bar\rho$ is $C^1$--conjugate into $\SO(2,\bR)$.
 Assume for each generator $s$ of $G$ and $x$ in the exceptional minimal set of $\rho$ we have that
\[ \rho(s)'(x) = \bar\rho(s)'\circ H(x).\]
Then we have that  \[\int_0^1 \alpha^{-1}(t)/t^{d+1}\;dt<\infty.\]
\end{thm}

The next result is a partial converse to Theorem~\ref{t:int-alpha}, which follows easily from Theorem~\ref{t:al-inv-intro} in the case of
a  single diffeomorphism.

\begin{thm}\label{t:alpha-inv}
Suppose $f$ is an exceptional $C^{1,\alpha}$ diffeomorphism, and suppose that $\ell_{i+1}/\ell_i\to 1$ as $i\to\infty$. Then there is a positive
constant $A>0$ such that \[\frac{A}{i}\leq \alpha(\ell_i)\] for all $i>0$.
\end{thm}

Theorem~\ref{t:alpha-inv} may be compared with Lemma 3.1.3 of~\cite{Navas2011}.
We immediately obtain the following consequence of Theorem~\ref{t:alpha-inv}, by a straightforward change of variables.
\begin{cor}\label{cor:int-alpha-inv}
Suppose $f$ is an exceptional $C^{1,\alpha}$ diffeomorphism, and suppose that $\ell_{i+1}/\ell_i\to 1$ as $i\to\infty$. Then
\[\int_0^1\frac{\alpha^{-1}(t)}{t^2}\, dt<\infty.\]
\end{cor}

Notice that since $\alpha$ is a homeomorphism of the non-negative reals, we may make sense of the notation $\alpha^{-1}$ as a function. 
Theorem~\ref{t:alpha-inv} recovers the fact that if $\ell_{i+1}/\ell_i\to 1$ as $i\to\infty$, then $f'$ cannot be $\alpha$--continuous for
$\alpha(x)=x\log{1/x}$ (see Exercise 4.1.26 and the examples in section 4.1.4 of~\cite{Navas2011}).

The final main result is as follows, and provides a partial converse to Theorem~\ref{t:al-inv-intro}.

\begin{thm}\label{t:int-comparison-intro}
Let $d$ be a positive integer, and let $\alpha$ be a concave modulus such that  
\[\int_0^1\frac{\alpha^{-1}(t)}{t^{d+1}}\, dt<\infty,\] and such that \[\sup_{t> 0}\frac{\alpha(t)}{t\alpha'(t)}<\infty.\]
If $\rho\co \bZ^d\to\Diff_+^{1,\alpha}(S^1)$ is a faithful representation  that is $C^1$--conjugate into the rotation group $\SO(2,\bR)$, then there exists a sequence 
\[\rho_k\co \bZ^d\to\Diff_+^{1,\alpha}(S^1)\]
of exceptional actions
which converges to $\rho$ in the $C^1$--topology.\end{thm}

\begin{rem}\label{rem:alpha-diff}
Here, the assumption that $\alpha$ is differentiable does not result in any loss of generality~\cite{Medvedev2001}. We also note that the
supremum in Theorem~\ref{t:int-comparison-intro} is bounded below by $1$ as follows from the standard concavity estimate
$t\alpha'(t)\leq\alpha(t)$.\end{rem}

The specialization of Theorem~\ref{t:int-comparison-intro} to a single diffeomorphism is as follows.

\begin{thm}\label{t:int-alpha-conv}
Let $\theta\in\bR\setminus\bQ$ be arbitrary. Suppose that \[\int_0^1\frac{\alpha^{-1}(t)}{t^2}\, dt<\infty,\] and that \[\sup_{t> 0}\frac{\alpha(t)}{t\alpha'(t)}<\infty.\]
Then there exists a $C^{1,\alpha}$ exceptional diffeomorphism $f$ with rotation number $\theta$, and such that $\ell_{i+1}/\ell_i\to 1$
as $i\to\infty$.
\end{thm}

Theorems~\ref{t:int-comparison-intro} and~\ref{t:int-alpha-conv} are close being converses to Theorems~\ref{t:c1al-intro}
and~\ref{t:int-alpha}, since the supremum hypotheses on $\alpha$ imply that $1/\alpha^d$ and $1/\alpha$ are integrable near zero,
 respectively.

\subsection{Denjoy counterexamples beyond moduli of continuity}

Identifying the precise conditions under which a even a single
homeomorphism $f\in\Homeo^+(S^1)$ with irrational rotation number is necessarily
topologically conjugate
to an irrational rotation is tantalizing. Some regularity of $f^{-1}$ is necessary, as was demonstrated by Hall~\cite{Hall81}. Even a
hypothetical characterization
 of moduli of continuity for which there exist exceptional diffeomorphisms does not appear to be the end of the discussion.
For instance, Sullivan proved that if the logarithm of the derivative of $f$ satisfies the Zygmund condition (also called the ``big"
Zygmund condition), then $f$ is topologically conjugate
to an irrational rotation~\cite{Sullivan1992}.
Note that the Zygmund condition for $\log f'$ implies that $f'$ is $\alpha$--continuous with $\alpha(x) = x\log 1/x$.

More generally, Hu and Sullivan~\cite{HuSu1997} show that if the derivative of $f$ has bounded quadratic variation and
bounded Zygmund variation, then $f$ is topologically conjugate to an irrational rotation. In the same paper, they show that if the
logarithm of the cross ratio distortion of $f$ has bounded variation, then $f$ cannot have any wandering intervals and hence must by
topologically conjugate to an irrational rotation.
None of these conditions on $f$ or $f'$ seem to be
expressible in terms of moduli of continuity for $f'$.

\subsection{Remarks on the structure of the paper}
Throughout this  paper, we strive for  generality insofar as  it  is possible. However,
we have taken steps to make the  exposition easier to follow for
readers who are not interested  in general group  actions but  rather in single diffeomorphisms.

Section~\ref{sec:blowup} discusses  generalities about finitely generated group actions on the circle,  blowups of such actions,
and constructions of exceptional diffeomorphism  actions by  finitely  generated groups. At first, general group actions are
discussed, though later the subexponential growth hypothesis comes into play, culminating in Corollary~\ref{c:stab0} and
the rather technical Theorem~\ref{t:c1al}.

Section~\ref{s:bounding} establishes Theorem~\ref{t:alpha-inv} and Theorem~\ref{t:int-alpha-conv}, which further clarify the  relationship
between the  lengths  of wandering  intervals and the possible moduli of continuity of  exceptional diffeomorphisms of the circle.
The content  of Section~\ref{s:bounding} applies to single diffeomorphisms and not to general  group actions.

Finally, Appendix~\ref{s:appendix} proves  Theorem~\ref{t:int-alpha} directly, without the added layers of 
technical difficulty coming from the relaxed
hypotheses of Theorem~\ref{t:c1al}. The content  of  the appendix again  applies  to  single diffeomorphisms  as opposed to general  group
actions, and  may be used by the reader both as a self-contained discussion and
 as  a  guide  to understand  the intuition behind the proof of Theorem~\ref{t:c1al}.

\section{Blowing-up finitely generated groups}\label{sec:blowup}
Throughout this section, we assume that
\[
\rho\co G\to\Homeo^+(S^1)\]
is an action of a finitely generated group $G$.
We often suppress $\rho$ in expressions involving group actions; that is, when the meaning is clear,  we will simply write
\[
gx=g(x) = \rho(g)(x), \quad
g'(x) =\rho(g)'(x).\]
The reader may imagine that $\rho$ is injective for the ease of reading.

Another action $\tilde \rho$ of $G$ is said to be \emph{semi-conjugate} to $\rho$ if there exists a
monotone surjective (hence continuous) degree--one map $H$ of $S^1$ such that
\[ H\circ \tilde\rho(g) = \rho(g)\circ H\]
 for all $g\in G$.
We sometimes call $H$ a \emph{semi-conjugacy map} from $\tilde\rho$ to $\rho$.

The semi-conjugacy is not an equivalence relation; however, one may define an equivalence
relation on the set of representations $\Hom(G,\Homeo_+(S^1))$
  by declaring that $\rho_0$ and $\rho_1$ are \emph{semi-conjugate} if they have a ``common blow-up''
  in a certain natural sense~\cite[Definition 2.3 and Theorem 2.2]{KKM2019}; see also~\cite{Ghys1999, CD2003IM, BFH2014}. 
It is essentially due to Poincar\'e that every countable group action is semi-conjugate to a minimal one, and also that a
semi-conjugacy preserves the rotation numbers.

If $H$ is a homeomorphism, then the actions $\tilde\rho$ and $\rho$ are conjugate; otherwise, we call $\tilde\rho$
 a \emph{nontrivial blow-up} of $\rho$ at the $\rho(G)$--equivariant set
\[
\{ x\in S^1\mid H^{-1}(x)\text{ is not a singleton }\}.\]
If $\rho$ does not admit a finite orbit, then every nontrivial blow-up of $\rho$ is  exceptional.

For a single minimal diffeomorphism of the circle, there is a standard technique of producing
nearby exceptional diffeomorphisms in the regularity $C^0$ due to Denjoy~\cite{Denjoy1932},
and also in $C^1$ and $C^{1,\alpha}$ due to Herman~\cite{Herman1979}; see
also~\cite{Tsuboi1995, DKN2007} for free abelian groups.
In this section, we will describe a common framework to produce exceptional actions of
finitely generated groups from a given one with various regularities.
This will generalize the aforementioned result of Herman. The reader is referred  to  the appendix  for a discussion of Herman's construction
in the case  of  a single diffeomorphism.

\subsection{Exceptional $C^0$--actions}
The following $C^0$--blow-up process is well-known~\cite{Denjoy1932,Herman1979}. 
We include a proof for the purpose of introducing some notation which will be useful in the sequel.
\begin{prop}\label{p:c0}
Let $\rho$ and $G$ be as above.
Then for each orbit $\mathcal{O}$ of $\rho$, there exists a sequence $\{\rho_k\}$ 
of blow-ups of $\rho$ at  $\mathcal{O}$
which converges to $\rho$ in $C^0$--topology. If $\rho$ is minimal then each $\rho_k$ is exceptional.
\end{prop}
\bp Let us fix a finite generating set $S$ of $G$.
For brevity, we may  assume \[\mathcal{O}=\rho(G).0.\]
We will choose the following parameters:
\begin{itemize}
\item a positive sequence $\{\ell_y\}_{y\in\mathcal{O}}$ such that $L:=\sum_{y\in\mathcal{O}}\ell_y\le 1$;
\item a family of orientation-preserving homeomorphisms 
\[\{\eta_{x,y,g}\co [0,\ell_x]\to[0,\ell_y] \mid {x,y\in\mathcal{O}}\text{ and }g\in G\text{ such that }y=gx\}\]
which is equivariant in the sense that 
\[
\eta_{gx,hgx,h}\circ\eta_{x,gx,g}=\eta_{x,hgx,hg}.\]
\end{itemize}
Note that we do \emph{not}  require  $\eta_{x,y,g}=\eta_{x,y,h}$ even when $gx=y=hx$; however, for the purpose of
the proof of the present proposition, the reader may  imagine that $\eta$ is linear:
\[\eta_{x,y,g}(t)=t\cdot \ell_y / \ell_x.\]

We will replace each orbit point $y\in\mathcal{O}\sse S^1$ by a closed interval of length $\ell_y$, whose interior will be denoted by
 $I_y$.  More formally, we define a measure
\[
d\lambda := (1-L)dm+\sum_{y\in \mathcal{O}} \ell_y d\delta_{y},\]
where $dm$ and $d\delta$ denote the Lebesgue and the Dirac measures, respectively. We let
\[
a_y:=\lambda[0,y),\ b_y:=\lambda[0,y],\]
and define $I_y:=(a_y,b_y)$ for $y\in\mathcal{O}$.

There exists a unique, natural, monotone, surjective, degree--one map
\[H\co S^1\to S^1\] satisfying $H(I_y)=\{y\}$ and $\mu(A)=m(H^{-1}A)$ for each Borel set $A$. That is, $\mu$ is the \emph{image}
(or push-forward) of Lebesgue measure by the map $H$.
Notice that we are treating the unit circle $S^1=\bR/\bZ$ as two different objects. One is the range of $H$, which is the original circle containing $\mathcal{O}$; the other is the domain of $H$, where each $I_y\in[0,1]$ is contained as an open interval.

Since $H$ is 1-1 on $K:=S^1\setminus\coprod_{y\in\mathcal{O}} I_y$, for each $g\in G$ we have a circle homeomorphism
\[
\tilde\rho(g)(t):=
\begin{cases}
a_{gy}+\eta_{y,gy,g}(t-a_y),&\text{ if }t\in\bar I_y,\\
H^{-1}\circ\rho(g)\circ H(t),&\text{ if }t\in K.\end{cases}\]
The equivariance assumption on $\eta$ implies that $\tilde\rho$ is a group action. If $\rho$ is minimal, then $K$ is a Cantor set and $\tilde\rho$ is exceptional. 

For each $k\in\bN$ we repeat the above process for the parameters $\{\ell_y^k\}_{y\in\mathcal{O}}$ with $\lim_k \ell_y^k=0$.
For instance, one may choose $\ell_y^k=\ell_y/k$
and arrange $\eta^k_{x,y,g}$ accordingly.
We let $\rho_k$ denote the resulting blow-up, with a semi-conjugacy map $H_k$. 

Let $\epsilon>0$. 
Since $H_k\to \Id$ uniformly,
 for each $s\in S$, and for all  $k\gg0$, we obtain 
\[
\|\rho_k(s) -  \rho(s)\|_\infty
<\|\rho_k(s) -  \rho(s)\circ H_k\|_\infty+\epsilon
=\|\rho_k(s) -  H_k\circ\rho_k(s)\|_\infty+\epsilon
<2\epsilon.\]
This completes the proof.\ep

\subsection{Exceptional $C^1$--actions}
We will now generalize Herman's construction for a single $C^1$--diffeomorphism to general finitely generated groups (cf. appendix of this
paper).
For this, we  follow the proof of Proposition~\ref{p:c0} while making a ``smoother'' choice of the equivariant family $\{\eta_{x,y,g}\}$.
For a countable set $A$ and a real function $f\co A\to\bR$,  we will write $\lim_{a\in A}f(a)=L$ if for an arbitrary enumeration $A=\{a_i\}$,
 we have $\lim_i f(a_i)=L$. The reader may simply imagine that $\{f(a_i)\}_i$ is  monotone decreasing in most cases.

\begin{thm}\label{t:c1}
Let $G$ be a finitely generated group, and let $\rho\co G\to\Diff_+^1(S^1)$ be an action with a choice of an orbit $\mathcal{O}$.
Assume that there exist positive sequences $\{\ell^k_y\}_{y\in\mathcal{O}}$ for $k\in\bN$ such that
\[\lim_{k\to\infty}\sum_{y\in\mathcal{O}} \ell_y^k=0\]
and such that
\[\lim_{k\to\infty} s'(y)\ell^k_y/ \ell^k_{sy}=1=\lim_{y\in \mathcal{O}} s'(y)\ell^k_y/ \ell^k_{sy}\]
 for each generator $s$ of $G$.
Then there exists a sequence  \[\{\rho_k\co G\to\Diff_+^1(S^1)\}\]
of blow-ups of $\rho$ at  $\mathcal{O}$
which converges to $\rho$ in the $C^1$--topology.
\end{thm}

We begin with defining a new equivariant family.
For $a,A>0$, we let
\[\psi(a,A)(t):=-\cot(\pi t/a)/A\co (0,a)\to\bR.\]
We then have 
\[\psi(a,A)^{-1}(t)=a/2 + a/\pi \cdot \arctan(At)\]
For $a,A,b,B>0$, we obtain a homeomorphism
\[
\phi(a,b,B/A):=\psi(b,B)^{-1}\circ\psi(a,A)\co [0,a]\to[0,b].\]
A special case of the family $\{\phi(a,b,b/a)\}_{a,b>0}$ is due to Yoccoz and
is used in~\cite{DKN2007,Navas2008GAFA}. See also~\cite{FF2003}, and ~\cite{Jorquera} for an improvement on Farb--Franks' result.
For each  $R>0$, we  define \[
\xi(R)(t):=  \frac{1+R}{1+R^2\cot^2(\pi t)}\cdot\chi_{[0,1]}(t) .\]
For a fixed $R>0$, the function $\xi(R)(x)$ on $x$ is a $C^1$--map supported on $[0,1]$.
\begin{lem}\label{l:Yoccoz}
The following hold for $a,b,c,R,S>0$.
\be
\item $\phi$ is equivariant: $\phi(b,c,S)\circ\phi(a,b,R)=\phi(a,c,RS)$.
\item $\int_\bR \xi(R)(t)dt=1$.
\item For all $t\in[0,a]$ we have that
\[
\phi(a,b,R)'(t) =\frac{b}{aR} \cdot (1-(1-R)\xi(R)(t/a)).\]
\ee
\end{lem}
\bp The proof is straightforward, and we omit the details.\ep

\bp[Proof of Theorem~\ref{t:c1}] 
For brevity, fix $k>0$ and let $\ell_y:=\ell^k_y$.
We let $H$ be the semi-conjugacy map determined by the parameter $\{\ell_y\}$ as in Proposition~\ref{p:c0}.
Set
\[
\lambda(g)(t):=g'\circ H(t)
\cdot
\left(1 - \sum_{y\in\mathcal{O}} \left(1 - \frac{\ell_{gy}}{g'(y)\ell_y}\right)\xi\left(\frac{\ell_{gy}}{g'(y)\ell_y}\right)\left(\frac{t-a_y}{\ell_y}\right)\right).\]
For each $t\in I_y=(a_y,b_y)$, we note that
\[\lambda(g)(t)= \phi\left(\ell_y,\ell_{gy},\frac{\ell_{gy}}{g'(y)\ell_y}\right)'
(t-a_y).\]
\begin{claim*}\label{c:c1-blowup}
For each $g\in G$, we have the following.
\be
\item\label{l:lam1}
The map $\lambda(g)\co  S^1\to\bR$ 
is positive and continuous for each $g\in G$.
\item\label{l:lam2}
For each $y\in \mathcal{O}$, we have
$\int_{I_y} \lambda(g) dm = m(I_{gy})$.
\item\label{l:lam3}
For each $x,y\in \mathcal{O}$, we have
$\int_{[a_x,a_y]} \lambda(g) dm = m[a_{gx},a_{gy}]$.
\item\label{l:lam4} $\int_{S^1} \lambda(g)dm = 1$.
\ee
\end{claim*}

\bp[Proof of the Claim]
(\ref{l:lam1})
The positivity of $\lambda$ is obvious. 
The uniform convergence theorem,
along with the estimate below
implies that $\lambda(g)$ is continuous:
\[
\lim_{y\in\mathcal{O}} 
\left|\left(
1 - \frac{\ell_{gy}}{g'(y)\ell_y}
\right)
\xi
\left(\frac{\ell_{gy}}{g'(y)\ell_y}\right)\left(
\frac{x-a_y}{\ell_y}\right)\right|
\le
\lim_{y\in\mathcal{O}} 
\left|
1 - \left(\frac{\ell_{gy}}{g'(y)\ell_y}\right)^2
\right|=0.
\]

(\ref{l:lam2})
Setting $R = (\ell_{gy}/\ell_y)/g'(y)$, we have
$
\int_{I_y}\lambda(g) =
\phi\left(\ell_y,\ell_{gy},R\right)(\ell_y)=\ell_{gy}$.

(\ref{l:lam3})
Recall we have set $K := S^1\setminus\coprod_w I_w$.
We note that
\[\int_{[a_x,a_y]\setminus K} \lambda(g) dm =
\sum_{w\in [x,y)\cap \mathcal{O}} \int_{I_w} \lambda(g)dm
=\sum_{w\in [x,y)\cap \mathcal{O}} \ell_{g(w)}
=\mu\left(g[x,y) \cap\mathcal{O}\right).\]
Recall $d\mu$ is the image of the Lebesgue measure $dm$.
A simple application of Riesz representation theorem as in~\cite[Theorem 1.19]{Mattila1995} shows that
\[
\int_{[a_x,a_y]\cap K} \lambda(g) dm =
\int_{H^{-1}([x,y)\setminus\mathcal{O} )} g'\circ H dm=\int_{[x,y) \setminus\mathcal{O}} g' d\mu
=\mu(g[x,y)\setminus\mathcal{O}).
\]
Summing up the above integrals, we obtain the desired conclusion. 
The proof of part (\ref{l:lam4}) is now immediate.
\ep

Consider the equivariant family
\[
\eta_{y,gy,g} := \phi\left(\ell_y,\ell_{gy},\frac{\ell_{gy}}{g'(y)\ell_y}\right)\co [0,\ell_y]\to[0,\ell_{gy}]\]
of $C^2$--diffeormophisms. 
From the claim and  from
\[
H(x) =\sup\{z\mid \mu[0,z)\le x\},\]
 we can see that  $\tilde\rho$ resulting from this equivariant family is expressed by
\[
\tilde\rho(g)(x) := a_{g({0})} + \int_{0}^x \lambda(g)\; dx \pmod\bZ.\]
Since $\lambda(g)$ is continuous, we have that  $\tilde\rho(G)\le\Diff_+^1(S^1)$.

We now repeat the above process for each $\{\ell_y^k\}_{y\in\mathcal{O}}$ and 
\[\eta_{x,y,g}^k:=
\phi\left({\ell_y^k},{\ell^k_{gy}},\frac{\ell^k_{gy}}{g'(y)\ell^k_y}\right)\co \left[0,{\ell^k_y}\right]\to\left[0,{\ell^k_{gy}}\right].\]
We denote  the resulting blow-up by $\rho_k$, and the semi-conjugacy map by $H_k$.
We let $I^k_y=H_k^{-1}(y)$ for each $y\in\mathcal{O}$.

It only remains to show that
\[
\|\rho_k(s)'-\rho(s)'\|_\infty\to 0\]
for each generator $s$ of $G$.
Let $\epsilon>0$ be arbitrary,
and let $k$ be sufficiently large.
For $t\in S^1\setminus\coprod_{y\in\mathcal{O}}I^k_y$, we see from the uniform continuity of $s'$  that
\[
|\rho_k(s)'(t)-\rho(s)'(t)|
= |s'\circ H_k(t)-s'(t)|<\epsilon.\]
In the case when $t\in \coprod_{y\in\mathcal{O}} I^k_y$,
we have as $k\to\infty$ that
\[
|\rho_k(s)'(t)-\rho(s)'(t)|
\le \|s'\|_\infty\cdot
\left|1 - \frac{\ell^k_{sy}}{s'(y)\ell^k_y}\right|
\cdot\|\xi\|_\infty
+|s'\circ H_k(t)-s'(t)|\to0,\]
which establishes the theorem.
\ep

\begin{rem}
We do \emph{not} require that $\mathcal{O}$ is a free orbit in Theorem~\ref{t:c1}.
If $gy=y$ and $g'(y)\ne1$ for some $y\in\mathcal{O}$, then 
$\tilde\rho(g)\restriction_{I_y}$ could be a nontrivial $C^1$--diffeomorphism acting on the interval $I_y$.
\end{rem}

\begin{rem}
In Theorem~\ref{t:c1},
we may simply assume the existence of a single positive sequence $\{\ell_y\}_{y\in\mathcal{O}}$ instead of $\{\ell_y^k\}_{y\in\mathcal{O}}$ for all $k$. In this case, we will only require the following:
\begin{itemize}
\item $\sum_{y\in\mathcal{O}}\ell_y<\infty$;
\item $\lim_{y\in \mathcal{O}} s'(y)\ell_y/ \ell_{sy}=1$ for each generator $s$;
\item $y\not\in I_y$ for each $y\in\mathcal{O}$,
where here $I_y$ is as defined in the proof of Proposition~\ref{p:c0} by the parameter $\{\ell_y\}$.
\end{itemize}
For $L=\sum_{y\in\mathcal{O}}\ell_y$, the last condition can be rephrased purely in terms of $\{\ell_x\}$,
namely: \[y\not\in\left(\sum_{x\in [0,y)\cap\mathcal{O}}\ell_x/L,\sum_{x\in [0,y]\cap\mathcal{O}}\ell_x/L\right).\]
Then we can simply put $\ell^k_y :=\ell_y/k$ and apply the above proof.
In this case,  each $t\in S^1$ satisfies the dichotomy that either
\[
t\in S^1\setminus\coprod_{y\in\mathcal{O}} I_y^k\]
for all sufficiently large $k$, or the following set is infinite:
\[
\{y\in\mathcal{O}\mid t\in I^k_y\text{ for some }k\}.\]
One can deduce the desired $C^1$--convergence from the second bullet point.
\end{rem}

Let $S$ be a fixed finite generating set of a group $G$.
For $g\in G$, we let $\|g\|_S$ denote the length of a shortest word in $S$ representing $g$. 
The \emph{spherical growth function} of $G$ is defined as
\[\sigma_{G,S}(n):= \#\{ g\in G\co \|g\|_S=n\}.\]
We often write $\|g\|:=\|g\|_S$ and $\sigma(n):=\sigma_{G,S}(n)$ when the meanings are clear from the context.
It is sometimes useful to consider the \emph{upper spherical growth function}
\[\bar\sigma(n):=\sup_{1\le i\le n}\sigma(i).\]
The functions $\sigma$ and $\bar\sigma$ are sub-multiplicative; moreover, $\bar\sigma$ is monotone increasing. 

Recall that a function $f\co\bN\to\bR_+$ is of \emph{sub-exponential growth} if
\[\limsup_{n\to\infty} f(n)^{1/n}=1.\]
The group $G$ is said to be of \emph{sub-exponential growth} if so is the map 
\[n\mapsto \#\{ g\in G\co \|g\|_S\le n\},\]
which is easily seen~\cite{dlHarpe2000} to be equivalent to each of the following:
\begin{itemize}
\item $\lim_{n\to\infty}\sigma(n)^{1/n}=1$;
\item $\lim_{n\to\infty} \bar\sigma(n)^{1/n}=1$.\end{itemize}

If  $\bar\sigma$ satisfies the following ``ratio--test'':
\[\lim_{n\to\infty} \bar\sigma(n+1)/\bar\sigma(n)=1,\] 
then $\bar\sigma$ is of sub-exponential growth.
A partial converse is the following.



\begin{lem}\label{l:sigma-upper}
If $f\co \bN\to[1,\infty)$ is a sub-multiplicative and monotone increasing function such that
\[\lim_{n\to\infty} f(n)^{1/n}=1.\]
Then there exists a monotone increasing function $g\co \bN\to[1,\infty)$  
such that 
\[
\lim_{n\to\infty} g(n+1)/g(n)=1\]
and such that each $n\in\bN$ satisfies
\[ f(n)\le g(n)\le f(n)^2.\]
\end{lem}
Note that we do not require $g$ is sub-multiplicative. 
\begin{exmp}
One may consider $f(n):=\exp\floor{\log_2 (n+1)}$, which can be seen to be sub-multiplicative,
monotone increasing and of sub-exponential growth.
The limit \[\lim_{n\to\infty}\frac{f(n+1)}{f(n)}\] does not exist, however. In this case, we set
\[g(n):=\exp({\log_2 (n+1)}),\]
which satisfies the desired properties in the above lemma.\end{exmp}
\bp[Proof of Lemma~\ref{l:sigma-upper}]
Let us consider the  nonnegative, sub-additive, monotone increasing function $F:=\log f$, so that
\[\lim_{n\to\infty} F(n)/n=0.\]
We will then define
\[G(n):=\frac1n\sum_{i=1}^n  F(i).\]

By the monotonicity of $F$, we have that $G$ is monotone increasing, and that $G(n)\le F(n)$. 
The sub-additivity of $F$ implies that 
\[F(n)-G(n) =\frac1n\sum_{i=1}^{n-1} \left(F(n)-F(i)\right)\le\frac{1}{n}\sum_{i=1}^{n-1}F(n-i) \le \frac{n-1}n G(n-1)\le G(n-1).\]


As $n\to\infty$ we finally have
\[G(n+1)-G(n)\le F(n+1)/(n+1)\to0.\]
We then see that $g(n):=\exp(2 G(n))$ is the desired function.\ep

Suppose we have an action $\rho\co G\to\Homeo^+(S^1)$, with a fixed base point $0\in S^1$. 
For each point $y\in\rho(G).0$, we define the \emph{orbit length of $y$ at $0$} as \[\|y\|_0:=\min\{j \mid s_j \cdots s_2 s_1(0)=y\text{ for some }s_i\in S\cup S^{-1}\}.\]
In the case the orbit $\rho(G).0$ is free, if $y=g(0)$ then we have $\|y\|_0=\|g\|_S$.
We can also define the \emph{orbit spherical growth function} analogously even when the orbit is not free; however, we will mainly deal with free orbits from now on.

\begin{cor}\label{c:stab0}
Let $G$ be a finitely generated group of sub-exponential growth.
If an action $\rho\co G\to \Diff_+^{1}(S^1)$ with a free orbit $\mathcal{O}:=\rho(G).0$ satisfies that
\[\sup_{g\in G}g'(0)<\infty,\]
then $\rho$ admits a sequence $\{\rho_k\co G\to\Diff_+^1(S^1)\}_{k\ge1}$ of nontrivial blow-ups at $\mathcal{O}$ such that
$\rho_k$ converges to $\rho$ in the $C^1$--topology.
\end{cor}
Note that the group $G$ is not assumed to be abelian.
\bp[Proof of Corollary~\ref{c:stab0}]
By applying Lemma~\ref{l:sigma-upper} to the sub-multiplicative function
\[
\bar\sigma(n)=\sup_{i\le n}\sigma(i)\]
we obtain a monotone increasing function $p(n)\ge1$ such that $p(n+1)/p(n)\to1$ as $n\to\infty$,
and such that
\[\bar\sigma(n)\le p(n)\le  \bar\sigma(n)^2.\]
Pick a continuous monotone increasing function $\nu(x)$ such that $\int_1^\infty p/ \nu<\infty$ and \[\lim_{x\to\infty}\nu(x+1)/\nu(x)=1;\] the reader may assume $\nu(x) = x^2p(x)$. 
For each $y\in \mathcal{O}$, there uniquely exists $g\in G$ such that $y=g(0)$. Then we set
\begin{equation}\label{eq:lky}\ell^k_{y} := g'(0) /\nu(\|y\|_0+k).\end{equation}
We obtain
\[
\sum_{y\in\mathcal{O}} 
\ell^k_y
\le \left(\sup_{g\in G} g'(0)\right)\cdot \sum_{i=0}^\infty  
\frac{p(i+k)}{\nu(i+k)}\cdot \frac{\bar\sigma(i)}{p(i+k)}
\to0\text{ as }k\to\infty.\]

For a generator $s$ of $G$,  and for  $g\in G$, we have
$\abs*{\|sg(0)\|_0-\|g(0)\|_0}\le 1$. This implies for $y=g(0)$ that
\[
\frac{s'(y)\ell_{y}^k}{\ell^k_{sy}}
= \frac{\nu(\|sy\|_0+k)}{\nu(\|y\|_0+k)}
\to 1\]
as $k+\|y\|_0\to\infty$. The conclusion follows from Theorem~\ref{t:c1}.
\ep

\begin{rem}
\be
\item
When $G\cong\bZ$, then the above corollary was proved by Herman~\cite[Theorem X.3.1]{Herman1979}.
In this case, the condition
$\sup_{g\in G}g'(0)<\infty$ holds possibly after replacing $0$ by some point $x_0\in S^1$, as was established by Ma\~n\'e. 
\item
If a sub-exponential growth group faithfully acts on $S^1$ with $C^{1+\epsilon}$ diffeomorphisms for some $\epsilon>0$, then the group is actually virtually nilpotent~\cite{Navas2008GAFA}.
\ee
\end{rem}

\subsection{Exceptional $C^{1,\alpha}$--actions}

We are now ready to restate and prove the first of our main results.

\begin{thm}\label{t:c1al}
Let $d\in\bN$, and let $\alpha$ be a concave modulus satisfying 
\[
\int_0^1 1/\alpha(x)^ddx<\infty.\]
If $d>1$, we further assume 
\begin{equation}\label{eq:alphay}
\sup_{0<y<1}  \alpha\!\left( y^{d+1}/\alpha(y)^d\right)/y<\infty.\end{equation}
Suppose that a  finitely generated group $G$ admits an action
$\rho\co G\to \Diff_+^{1,\alpha}(S^1)$ with a free orbit $\mathcal{O}:=\rho(G).0$ such that  \[\sup_{g\in G} \abs*{\log g'(0)}<\infty.\]
Suppose furthermore that the spherical growth function of $G$ is at most $c n^{d-1}$ for some $c>0$.
Then $\rho$ admits a sequence of nontrivial blow-ups 
\[\{\rho_k\co G\to\Diff_+^{1,\alpha}(S^1)\}\]
at $\mathcal{O}$
such that $\rho_k\to\rho$ in the $C^1$--topology.
Moreover, we may require that
\[\sup_k\{[\rho_k(g)]_{\alpha}\}<\infty\] for all $g\in G$.
\end{thm}

\begin{rem}
The proof of Theorem~\ref{t:c1al} is rather technical, and the reader interested  in single exceptional
diffeomorphisms may wish to dispense with the
generality afforded by assuming $d\neq 1$, as well as more general group actions.
For the benefit of such readers, we have included a proof of
Theorem~\ref{t:int-alpha} in Section~\ref{s:appendix} below.
This proof also serves as a detailed guide for understanding the proof of Theorem~\ref{t:c1al}.
\end{rem}

Note that a group $G$ satisfying the hypotheses of Theorem~\ref{t:c1al}
 is necessarily virtually nilpotent, by Gromov's theorem on groups of polynomial word growth. The reader
may wonder what the precise difference between the spherical growth function and the \emph{word growth function} is,
the latter of which measures the number of words of length at most $n$ with respect to some generating set,
and whose degree is independent of the generating set (cf.~\cite{mann-book}, for instance, for background).

If $G$ is a finitely generated abelian group of rank $d$ then the word growth function is polynomial of degree $d$, and the spherical
growth  function is polynomial of degree $d-1$, as is an easy exercise. For a general group $G$ of polynomial word growth, Gromov's
theorem allows us to assume that $G$ is torsion--free and nilpotent, and that if $\gamma_i(G)$ denotes the $i^{th}$ term of the lower central
series of $G$, then $\gamma_i(G)/\gamma_{i+1}(G)$ is torsion--free for all $i$. In this case, for groups of nilpotence class two (i.e.
$\gamma_3(G)=\{1\}$, where here we adopt the convention $\gamma_1(G)=G$), it is also true that if the degree of the word growth
function is $d$ then the spherical growth function has degree $d-1$. This fact can be established by explicit computation~\cite{Stoll1998}; see the proof
of Theorem 4.2 in~\cite{mann-book}. In general, one has a bound which appears to be related in a subtle
way to the error term in computing the exact word growth function. Namely:

\begin{prop}[See~\cite{BLD}, Corollary 11]
Let $G$ be a nilpotent group such that $\gamma_{r+1}(G)=\{1\}$
and with polynomial word growth of degree $d$ with respect to a generating set $S$.
Then there are constants $C_1$ and $C_2$, depending on $S$, such that
\[C_1 n^{d-1}\leq \sigma(n)\leq C_2 n^{d-\epsilon},\] where $\epsilon=1$ for $r=2$ and $\epsilon=2/3r$ otherwise.
\end{prop}

Thus, even for general finitely generated nilpotent group, the precise relationship between the word growth function and the spherical
growth function seems subtle, and thus we elect to state our  results in terms of the spherical growth function as opposed to the more
common word growth function.

We return to the discussion of Theorem~\ref{t:c1al}.
In the case of minimal actions, the boundedness of logarithms of derivatives gives a strong restriction on possible actions:
\begin{prop}[{cf. \cite[Chapter IV]{Herman1979} for the case $\Gamma=\bZ$}]\label{p:GoHe}
Let $\Gamma$ be a subgroup of $\Diff_+^1(S^1)$.
If the action of $\Gamma$ is minimal, and if 
 \[\sup_{g\in \Gamma} \abs*{\log g'(x_0)}<\infty\]
for some $x_0\in S^1$, then $\Gamma$ is $C^1$--conjugate into the rotation group.
\end{prop}

For the proof of this proposition, we employ an equivariant version of the \emph{Gottschalk--Hedlund Lemma}.
For a group $\Gamma$ acting on a space $M$, a \emph{continuous 2--cocycle} is a map
\[c\co \Gamma\times M\to\bR\]
such that $x\mapsto c(g,x)$ is continuous,
and such that  \[c(fg,x) = c(g,x)+c(f,gx).\]

\begin{lem}[{\cite[Lemma 3.6.10]{Navas2011}}]
Suppose that a group $\Gamma$ acts minimally on a compact metric space $M$.
Assume that $c\co \Gamma\times M\to\bR$ is a continuous 2-cocycle.
Then the following are equivalent.
\be[(i)]
\item We have $\sup_{g\in\Gamma} |c(g,x_0)|<\infty$ for some $x_0\in M$;
\item There exists a continuous map $\phi\co M\to\bR$ such that 
\[c(g,x) = \phi(x) - \phi\circ g(x)\] for all $g\in\Gamma$ and $x\in M$.
\ee
\end{lem}

\begin{rem}
In~\cite{Navas2011}, the above lemma is stated with a finite generation hypothesis for $\Gamma$. However, it is apparent from the proof therein that this hypothesis is unnecessary. See~\cite{CNP13} for a further generalization of the lemma.
\end{rem}

\bp[Proof of Proposition~\ref{p:GoHe}]
Applying the Gottschalk--Hedlund Lemma to 
\[c(g,x) := \log g'(x),\]
we obtain a continuous map $\phi\co S^1\to\bR$ such that 
\[\log g'(x) = \phi(x) - \phi\circ g(x).\]
As in~\cite[Chapter IV]{Herman1979}, we define $h\in\Diff_+^1(S^1)$ by
\[
h(x) :=\int_0^x e^{\phi(t)+c}dt\pmod\bZ.\]
Here, $c$ is chosen so that $h(1) = 1$. For each $g\in \Gamma$, we can verify
\[h'\circ g(x) \cdot g'(x) = e^{\phi\circ g(x)+c}\cdot e^{\phi(x)-\phi\circ g(x)}=e^{\phi(x)+c}=h'(x).\]
This implies that $(hgh^{-1})'=1$, so that $hgh^{-1}\in\SO(2,\bR)$.
\ep

\bp[Proof of Theorem~\ref{t:c1al}]
As we have noted in Remark~\ref{rem:alpha-diff}, we may assume that $\alpha$ is differentiable. 
We retain the notation used in Corollary~\ref{c:stab0}, setting
\[p(x):=x^{d-1},\quad \nu(x) := x^{d+1}\alpha(1/x)^d.\]
Note that $\int_1^\infty p/\nu = \int_0^1 1/ \alpha^d<\infty$.
By convexity, we see that $x/\alpha(x)$ is monotone increasing, and so is $\nu(x)/x$. 

We define $\ell_y^k$ as in the equation~\eqref{eq:lky}.
In the proof of Corollary~\ref{c:stab0}, we observed that the hypothesis $\int_1^\infty p/\nu<\infty$ implies that \[\sum_{y\in\mathcal{O}}\ell_y^k\to0.\]
For each generator $s$ of $G$, each point $y\in \mathcal{O}$, and each $k\gg0$, we define
 \[A(s,k,y):=\frac1{\alpha(\ell_y^k)}\abs*{1-\frac{\ell^k_{sy}}{s'(y)\ell_y^k}}.\]
Let  $\lesssim$ denote an inequality up to bounded errors and scaling. 

\begin{claim*}
$\{A(s,k,y)\}_{s,k,y}$ is uniformly bounded.
\end{claim*}
We remind the reader that in Appendix~\ref{s:appendix} , there is  a proof of the theorem in the case $d=1$ and  $G=\bZ$.
The  reader will find the analogous claim in  that proof, and may find it helpful to understand the workings of the present argument.

To see this claim in the given generality, we first note \[\alpha(cx)\ge  \min(1,c)\alpha(x),\] for $c,x>0$. 
Letting $j:=\|sy\|_0+k$, we have 
 \[
A(s,k,y)
\lesssim\frac{1}{\alpha(1/\nu(j))}\abs*{1- \frac{\nu(j)}{\nu(j\pm1)}}
\le \frac{ \nu'(t)}{\nu(j-1)\alpha(1/\nu(j))}.
\]
Here, $t\in(j-1,j+1)$ is chosen so that \[\nu'(t) = \nu(j\pm1)-\nu(j).\]
We compute
\[
x\nu'(x)=\left(d+1 - d\alpha'(1/x)/(x\alpha(1/x))\right)\nu(x)\in [\nu(x),(d+1)\nu(x)].
\]
Also, note that \[\nu(x/2)= (x/2)^{d+1}\alpha(2/x)^d\ge \nu(x)/2^{d+1}.\]
Assuming $k\gg0$, we see
\[ 
A(s,k,y)\lesssim \frac{ \nu(j+1)}{j\nu((j+1)/2)\alpha(1/\nu(j))}\lesssim\frac{1}{j\alpha(1/\nu(j))}. \]
If $d\ge 2$, then the equation \eqref{eq:alphay} implies the claim.
If $d=1$, then we recall that $\nu(x)/x$ is increasing and that
\[
\frac{1}{j\alpha(1/\nu(j))}=\frac{\nu(j)^2}{j\nu\circ\nu(j)}\le1.
\]
So the claim is proved.

Note that the claim trivially implies that 
$\lim_{x\to\infty}\nu(x+1)/\nu(x)=1$. Hence, we have actions $\{\rho_k\}$ satisfying the conclusion of Corollary~\ref{c:stab0}.

It only remains to show that $\sup_k [\rho_k(s)']_\alpha<\infty$ for each generator $s$ of $G$.
We recall the construction of $\rho_k$ from the proof of Theorem~\ref{t:c1}.
Suppose first that $u,v\in\bar I_y$ for some $y\in \mathcal{O}$.
Then we have
\[
\abs*{\rho_k(s)'(u)-\rho_k(s)'(v)}
\le \abs*{s'(y)} \cdot 
\abs*{1-\frac{\ell^k_{sy}}{s'(y)\ell_{s}^k}}
\cdot \|\xi\|_\infty\cdot \frac{|u-v|}{\ell_y^k}.\]
Since $|u-v|\le\ell_y^k$, we have that
\[
|u-v|/\ell_y^k \le \alpha(|u-v|)/\alpha(\ell_y^k).\]
The above claim implies that 
\[ \abs*{\rho_k(s)'(u)-\rho_k(s)'(v)}\le C \alpha(|u-v|)\]
for some constant $C$ independent of $s,k$ and $y$. Here, we are using the finite generation hypothesis. 

Let us now assume $u,v\not\in\coprod_{y\in\mathcal{O}} I_y$. 
We recall the definition of $\mu_k$, given by \[m(H_k^{-1}(A))=\mu_k(A)\] for each Borel set $A$.
For $k\gg0$ and \[L_k=\sum_{y} \ell_y^k<1/2,\] we see
\[ |H_k(u)-H_k(v)|\le  |u-v|/ (1-L_k)\le 2 |u-v|.\]
So, by enlarging $C$ if necessary for all $k\gg0$ we have
\begin{align*}
\abs*{\rho_k(s)'(u)-\rho_k(s)'(v)}
&= \abs*{ s'\circ H_k(u)-s'\circ H_k(v)}
\le [s']_\alpha  \alpha( H_k(u)-H_k(v))\\
&\le C \alpha(u-v).\end{align*}
Finally, we conclude that $[\rho_k(s)']_\alpha\le 3 C$.
\ep


\subsection{Applications to free abelian groups and Theorem~\ref{t:int-alpha}}
Let us note immediate corollaries of Theorem~\ref{t:c1al} for free abelian groups.

\begin{cor}\label{c:higher-rank}
Let $d\ge1$, and let $\alpha$ be a concave modulus such that \[\int_0^1\frac{1}{\alpha^d}\, dx<\infty.\]
If $d>1$, we further assume 
\[\sup_{0<y<1}  \alpha\!\left( y^{d+1}/\alpha(y)^d\right)/y<\infty.\]
If $\rho\co \bZ^d\to\Diff_+^{1,\alpha}(S^1)$ is a faithful representation  that is $C^1$--conjugate into the rotation group $\SO(2,\bR)$, then there exists a sequence 
\[\rho_k\co \bZ^d\to\Diff_+^{1,\alpha}(S^1)\]
of nontrivial blow-ups of $\rho$ such that $\rho_k\to\rho$ in the $C^1$--topology.
\end{cor}

Note that the second inequality above holds for the modulus 
\[\alpha(x) =\left(x \cdot \log 1/x\right)^{1/d}\cdot (\log 1/x)^\epsilon.\]
for $\epsilon>0$ as  considered by
Deroin--Kleptsyn--Navas~\cite{DKN2007}, and for the stronger modulus
\[\alpha(x) =\left(x\cdot  \log 1/x\cdot \log\log1/x\right)^{1/d}\cdot (\log\log 1/x)^\epsilon.\]

Theorem~\ref{t:int-alpha} is now an immediate consequence.
We remark that the method of this section, however, does not extend to the case when $\alpha(x) = x\log (1/x)$;
cf.~\cite[Exercise 4.1.26]{Navas2011} and~\cite[Chapter 12]{KH1995}.

\section{Bounding lengths from below}\label{s:bounding}
In this section, we establish Theorem~\ref{t:alpha-inv} and Theorem~\ref{t:int-alpha-conv}.
For this, we will let $\alpha$ be a given concave modulus, and let  \[\rho\co G\to\Diff_+^{1,\alpha}(S^1)\]
be an exceptional action. We write $K$ for the exceptional minimal set of this action, which is the unique, closed, minimal, invariant
subset of $S^1$ under the action of $G$. The complement of $K$ is a union of open intervals in $S^1$, which are permuted
component-wise by the action of $G$, and thus consist of wandering intervals for the action of $G$ on $S^1$. There is a semi-conjugacy
$H\colon S^1\to S^1$ which collapses each wandering interval to a point, and which is at most two--to--one on $K$. The wandering
intervals for the action can obviously be partitioned into at most countably many orbits under the action of $G$. Each endpoint of a wandering
interval lies in $K$, and the action of $K$ is minimal. It follows that if $I$ and $J$ are two wandering intervals, not necessarily in the same
orbit under $G$, then $G.(\partial I)$ accumulates on $\partial J$.

The reader will note that we lose no generality in what follows by assuming that there is only one $G$--orbit of wandering intervals for
the action of $G$, and we will thus adopt this assumption. In this case, the inverse of $H$ is a blow-up of exactly one orbit of a minimal
action, say $\mathcal{O}$.

Each point $y\in\mathcal{O}$ corresponds to a maximal wandering interval \[I_y=(a_y,b_y)\sse S^1.\] 
We will continue to denote the minimal set by \[K:=S^1\setminus \coprod_{y\in\mathcal{O}}I_y.\]

We let $\|y\|_0$ and $\sigma(n)$ denote the orbit distance and the  spherical growth functions as before, and put $\ell_y:=|I_y|$. Note that we have
\[\ell_{gy}=|I_{gy}|=|gI_y|,\]
where we mean $gy=\bar\rho(g)y$ and $gI_y=\rho(g)I_y$.

\subsection{Endpoint derivatives and the fundamental estimate}
The following lemma is well-known, and is crucial for the content of this section.

\begin{lem}[Fundamental Estimate]\label{lem:fundamental}
Let $\alpha,G,\rho,\mathcal{O},\{\ell_y\}$ be as above, and let $g\in G$. If
\[\lim_{y\in\mathcal{O}} |\ell_{gy}|/|\ell_y|=1,\]
then we have that \[\sup_{y\in\mathcal{O}} \frac{1}{\alpha(\ell_y)}\abs*{1-\frac{\ell_{gy}}{\ell_{y}}}<\infty.\]
\end{lem}
Recall the limit above is taken with respect to an arbitrary enumeration of $\{\ell_y\}$. It is instructive to imagine that $\{\ell_y\}$ is listed in the monotone decreasing order. The reader will note that the limit in the hypotheses of Lemma~\ref{lem:fundamental} continues to make sense
even if there are many orbits of wandering intervals for the action of $G$. Indeed, the limit compares lengths of wandering
 intervals in a single orbit only, so no comparisons between distinct orbits need to be made.

\begin{proof}[Proof of Lemma~\ref{lem:fundamental}]
We first claim that $g'(x)=1$ for all $x$ in the minimal set $K$. 
By The Mean Value Theorem, for each $y\in \mathcal{O}$ there exists  $z_y\in I_y$ satisfying $g'(z_y)=\ell_{gy}/\ell_y$.
In the case when $x\in  I_y$ for some $y\in \mathcal{O}$, 
then by minimality we obtain a sequence $\{y_i\}\in\mathcal{O}\setminus\{y\}$ such that $\lim_i  I_{y_i}\to x$. 
It follows that  
\[
g'(x) = \lim_i g'(z_{y_i}) = \lim_i \ell_{gy_i}/\ell_{y_i}=1
\]
Since $\{\partial I_y\}_{y\in\mathcal{O}}$ is dense in $K$, our claim is proved.

Let us now fix $y\in\mathcal{O}$, and $w\in\partial I_y$.
The conclusion  follows from  the inequality
\[\abs*{1-\ell_{gy}/\ell_y}
=
\abs*{s'(w)-s'(z_y)}
\leq \brac*{g'}_\alpha\cdot \alpha(\ell_y).\qedhere\]
\end{proof}

\subsection{Bounding the values of $\alpha(\ell_y)$}
It is evident that passing from the fundamental estimate (Lemma~\ref{lem:fundamental}) to \emph{a priori} bounds on the
values of $\{\ell_y\}_{y\in\mathcal{O}}$ is a natural approach to a converse to Theorem~\ref{t:c1al}, and this is precisely the sort of
control given to us in Theorem~\ref{t:alpha-inv} above. Recall the following result from the introduction.

\begin{thm}\label{t:al-inv}
Let $\alpha$ be a concave modulus, and let 
\[\rho\co \bZ^d\to \Diff_+^{1,\alpha}(S^1)\]
be an exceptional action. 
Suppose that $\rho$ is semi-conjugate to a faithful action 
\[\bar\rho\co\bZ^d\to\Diff_+^1(S^1)\]
by a semi-conjugacy map  $H$ such that $\bar\rho$ is $C^1$--conjugate into $\SO(2,\bR)$.
 Assume for each generator $s$ of $G$ and $x$ in the exceptional minimal set of $\rho$ we have that
\[ \rho(s)'(x) = \bar\rho(s)'\circ H(x).\]
Then we have that  \[\int_0^1 \alpha^{-1}(t)/t^{d+1}\;dt<\infty.\]
\end{thm}

\bp
Set $G:=\bZ^d$.
Let $J$ be a maximal open wandering interval of $\rho$, and let $\mathcal{O}$ be the $\bar\rho(G)$--orbit of $0:=H(J)$.
We may write
\[
\{gJ\}_{g\in G}=\{I_y:=(a_y,b_y)\}_{y\in\mathcal{O}}\]
and $\ell_y:=I_y$. 

For each $y=\bar\rho(g).0\in\mathcal{O}$,  there exists $z_y\in I_y$ such that  \[\rho(s)'(z_y)=\ell_{sy}/\ell_y.\]
Let $s$ be a generator of $G$. Then we have that
\begin{align*}
\brac*{\rho(s)'}_\alpha \cdot\alpha(\ell_y)
&\ge \abs*{\rho(s)'(z_y)-\rho(s)'(a_y)}
\ge \abs*{\frac{\ell_{sy}}{\ell_y}-\bar\rho(s)'(y)}
\\
&\ge \inf_{x\in S^1} \bar\rho(s)'(x) \cdot \abs*{\frac{\ell_{sy}}{\bar\rho(s)'(y)\ell_y}-1}.\end{align*}
Setting $L_y :=\ell_y / \bar\rho(g)'(0)$, we can find a constant $C>0$ such that
\[
\abs*{\frac{L_{sy}}{L_y}-1} \le C\alpha(L_y).\]
Here, we also used the fact
\[
\alpha(\ell_y)\le \max(1,\sup_{h\in G} \bar\rho(h)'(0)) \alpha(L_y).\]

It is easy to show that $h(x)=x(1-C\alpha(x))$ is increasing for small values of $x$, and that in fact $h'(0)\geq 1$.
Recall we let $\|y\|_0$ denote the orbit distance between $0$ and $y$.

\begin{claim*}
There exists a constant $A>0$ such that $\alpha(\ell_y)\ge A/\|y\|_0$.
\end{claim*}
The idea is similar to \cite[Lemma 3.1.3]{Navas2011}, where the case $\alpha(x)=x^{1/d}$ was considered.
We proceed by induction on $i:=\|y\|_0$. 
We suppose that $\alpha(\ell_y)\geq A/\|y\|_0$. After possibly requiring $\ell_y\le\epsilon$ for some small $\epsilon>0$, we obtain
\[\ell_{sy}\geq h(\ell_y)\geq h(\alpha^{-1}(A/\|y\|_0))=\alpha^{-1}\left(\frac{A}{\|y\|_0}\right)\left(1-\frac{AC}{\|y\|_0}\right),\] where the second inequality
holds because $\alpha$ is a homeomorphism and because $h$ is increasing near zero. We consider $s,y$ such that $\|sy\|_0=\|y\|_0+1=i+1$. We apply $\alpha$ to this string of inequalities now. 
Since $\alpha$ is concave, we have that for
any $0<c<1$, there is an inequality  $\alpha(cx)\geq c\alpha(x)$. Thus, after applying $\alpha$ to the right--most hand side, we get a quantity
bounded below by \[\left(1-\frac{AC}{i}\right)\frac{A}{i}.\] If one can show that this quantity is bounded below by $A/(i+1)$, then one obtains
$\alpha(\ell_{sy})\geq A/\|sy\|_0$, the desired conclusion. A straightforward manipulation shows that we obtain 
\[\left(1-\frac{AC}{i}\right)\frac{A}{i}\geq \frac{A}{i+1}\] provided that \[i\geq \frac{CA}{1-CA}.\] Thus, letting $A$ be a positive constant so
that $\alpha(\ell_y)\geq A/\|y\|_0$ even for $\ell_y\ge\epsilon$ and so that \[1\geq \frac{CA}{1-CA},\] we obtain the desired claim.

From the claim and from $\sigma(n)\sim n^{d-1}$, we have that
\[
1\ge  \sum_{y\in\mathcal{O}} \ell_y
\gtrsim \sum_{n\ge1} n^{d-1} \alpha^{-1}(A/n).\]
We have that 
\[
\infty > \sum_{n\ge1} (n+1)^{d-1} \alpha^{-1}(A/n)
\ge \int_1^\infty x^{d-1} \alpha^{-1}(A/x)dx.\]
This implies that $\int_0^1 \alpha^{-1}(t)/t^{d+1} \; dt<\infty$.
\end{proof}

Note that the claim in the proof of Theorem~\ref{t:al-inv} immediately implies Theorem~\ref{t:alpha-inv} from the introduction.
In Theorem~\ref{t:al-inv}, if $\bar\rho$ is actually a representation into the rotation group, then it suffices to assume 
\[\lim_{y\in\mathcal{O}}\ell_{sy}/\ell_y=1,\]
where $\ell_y$ is as given in the proof. 
From the Fundamental Estimate (Lemma~\ref{lem:fundamental}),
we conclude that $\rho(g)'(x)=1$ for each $x$ in the exceptional minimal set of $\rho$, and hence the same integrability constraints on 
$\alpha$ would follow.
We now have  Theorem~\ref{t:alpha-inv}
and Corollary~\ref{cor:int-alpha-inv} as  immediate consequences.

\begin{que}
In Theorem~\ref{t:al-inv}, does the same conclusion hold without the assumption on the values of $\rho(s)'$ at the minimal set?\end{que}

\begin{exmp}
Let $f$ be an exceptional diffeomorphism such that $\ell_{i+1}/\ell_i\to 1$ for intervals in a maximal wandering set. In this case, let us deduce from Theorem~\ref{t:alpha-inv} that $f'$ is not
$\alpha$--continuous for $\alpha(x)=x\log{1/x}$ (cf.~\cite{Navas2011,KH1995}).


Suppose $f'$ is $\alpha$--continuous. By Theorem~\ref{t:alpha-inv}, there exists a positive constant $A$ such that $\alpha(\ell_i)\ge A/i$ for all $i$. 
On the other hand, we have a positive sequence $\{x_i\}_{i\geq 2}$ defined below such that $\sum_i x_i=\infty$ and such that $\alpha(x_i)\leq A/i$. This is a contradiction, since we have
\[\ell_i \ge \alpha^{-1}(A/i) \ge x_i.\]

The sequence $\{x_i\}$ is actually defined as \[x_i=\frac{A}{K\cdot i\log i},\] where $K$ is a positive constant which will be fixed later.
Observe that $\sum_i x_i=\infty$, so it suffices to show that $\alpha(x_i)\leq A/i$ for all $i\geq 2$. Applying $\alpha$, we see
\[\alpha(x_i)=\frac{A}{K\cdot i\log i}\left( \log K+\log i+\log\log i+\log{A^{-1}}\right).\] This simplifies to
\[\frac{A(\log K+\log{A^{-1}})}{K}\cdot\frac{1}{i\log i}+\frac{A}{K\cdot i}+\frac{A\cdot \log\log i}{K\cdot i\log i}.\] Since $1/(i\log i)<1/i$ for $i\geq 3$,
since $(\log\log i)/\log i<1$ for $i\geq 2$,
and since $A$ is a fixed positive constant, it is clear that we can choose a sufficiently large $K$ to make this entire expression less than $A/i$
for $i\geq 2$, which is what we set out to show.

It is not difficult to generalize this last computation to show the same conclusion for the moduli \[\alpha(x)=x(\log{1/x})(\log\log{1/x}),\]
\[\alpha(x)=x(\log{1/x})(\log\log{1/x})(\log\log\log{1/x}),\] and indeed for
\[\alpha(x)=x(\log{1/x})(\log\log{1/x})(\log\log\log{1/x})\cdots(\log^n{1/x}).\]
\end{exmp}

\subsection{A converse to Theorem~\ref{t:alpha-inv}}
We now investigate the degree to which \[\int_0^1 \alpha^{-1}(t)/t^{d+1} dt<\infty\]
is a sufficient condition for the existence of an exceptional $C^{1,\alpha}$ 
action of $\bZ^d$.
We recall the following result, which was stated in the introduction.

\begin{thm}\label{t:int-comparison}
Let $d$ be a positive integer, and let $\alpha$ be a concave modulus such that  
\[\int_0^1\frac{\alpha^{-1}(t)}{t^{d+1}}\, dt<\infty,\] and such that \[M:=\sup_{0<t<1}\frac{\alpha(t)}{t\alpha'(t)}<\infty.\]
If $\rho\co \bZ^d\to\Diff_+^{1,\alpha}(S^1)$ is a faithful representation  that is $C^1$--conjugate into the rotation group $\SO(2,\bR)$, then there exists a sequence 
\[\rho_k\co \bZ^d\to\Diff_+^{1,\alpha}(S^1)\]
of nontrivial blow-ups of $\rho$ 
that converges to $\rho$ in the $C^1$--topology.\end{thm}
\bp
We follow our previous arguments for Theorems~\ref{t:c1} and~\ref{t:c1al} closely. We let $\mathcal{O}=\rho(G).0$ be a choice of a free orbit. 
Using the  spherical growth $\sigma(x)\sim x^{d-1}$ and the orbit distance function $\|y\|_0$, 
we set the parameters \[\ell_{g(0)}^k:=g'(0)\alpha^{-1}(1/ (\|y\|_0 + k)).\]
Since $\alpha^{-1}(t)/t^{d+1}$ is integrable near $0$, we have that
 \[\sum_{y\in\mathcal{O}}\ell_y^1 \lesssim \sum_n \sigma(n)\alpha^{-1}(1/n)\lesssim \int_1^\infty x^{d-1}\alpha^{-1}(1/x)dx=\int_0^1 \alpha^{-1}(t) / t^{d+1} \; dt <\infty.\]

\begin{claim*}
We have \[\lim_{x\to\infty} \frac{\alpha^{-1}(1/x)}{\alpha^{-1}(1/(x+1))}=1.\]
\end{claim*}
To see the claim, we put
\[ c(x):= \alpha^{-1}(1/x) / \alpha^{-1}(1/(x+1))\ge1\] for $x\gg0$. Then we have some $t=t(x)\in(x,x+1)$ such that 
\begin{align*}
& 0< c(x) - 1 =  \frac{(\alpha^{-1})'(1/t)}{x(x+1)\alpha^{-1}(1/(x+1))}
=\frac{1}{x(x+1) \alpha^{-1}(1/(x+1)) \cdot \alpha'\circ\alpha^{-1}(1/t)}
\\
&\le
\frac{M\alpha^{-1}(1/x)}{x\alpha^{-1}(1/(x+1))}=M c(x) / x.\end{align*}
The claim follows from the estimate $c(x) <(1- M/x)^{-1}$   for $x\gg0$.

Put $j:=\|y\|_0+k$ for $y\in\mathcal{O}$ and $k>0$.
As in the proof of Theorem~\ref{t:c1al}, 
 \[ A(s,k,y):=\frac1{\alpha(\ell_y^k)}\abs*{1-\frac{\ell^k_{sy}}{s'(y)\ell_y^k}}
 \le j \abs*{1- \frac{\alpha^{-1}(1/(j\pm1))}{\alpha^{-1}(1/j)}}
\lesssim\frac{1}{j\alpha'\circ\alpha^{-1}(1/t)\cdot \alpha^{-1}(1/j)}
\]
for some $t\in(j-1,j+1)$.
Using the hypothesis on $\alpha(x)/(x\alpha'(x))$, we estimate
\[
A(s,k,y)  \lesssim
\frac{\alpha^{-1}(1/t)}{\alpha^{-1}(1/j)},\]
up to scaling and bounded error.
By the above claim, $A(s,k,y)$ is uniformly bounded as desired.
\ep

Theorem~\ref{t:int-alpha-conv} follows immediately from Theorem~\ref{t:int-comparison}.


\begin{rem}
In the case when $d=1$, one could attempt to prove a converse to Theorem~\ref{t:alpha-inv} by starting with the sequence $\{A/i\}_{i>0}$ and attempting to find
a convergent sequence of interval lengths $\{\ell_i\}_{i\in\bZ}$ such that \[\sup_i\frac{1}{\alpha(\ell_i)}\left(1-\frac{\ell_{i+1}}{\ell_i}\right)=K\]
is finite, whereupon one can apply the method of Theorem~\ref{t:c1al} in order to find the desired exceptional diffeomorphism. A natural
choice of lengths would be \[\ell_i=\alpha^{-1}\left(\frac{A}{|i|+i_0}\right),\] where here $i_0\in\N$ is a suitable shift of indices.

With this choice, the finiteness of the supremum forces the following inequality to hold:
\[\alpha^{-1}\left(\frac{A}{|i|+i_0+1}\right)\geq \alpha^{-1}\left(\frac{A}{|i|+i_0}\right)\left(1-\frac{KA}{|i|+i_0}\right),\] which is equivalent to
\[\frac{A}{|i|+i_0+1}\geq \alpha\left(\alpha^{-1}\left(\frac{A}{|i|+i_0}\right)\left(1-\frac{KA}{|i|+i_0}\right)\right).\] One can then try and pull out
the constant $1-(KA)/(|i|+i_0)$ on the right hand side, but because this constant is less than one, the result will be less than or equal to
the right hand side. In order to get around this difficulty, one can attempt a linear approximation of $\alpha$ at $\alpha^{-1}(A/(|i|+i_0))$
in order to show that \[\frac{A}{|i|+i_0+1}\geq \left(1-\frac{KA}{|i|+i_0}\right)\left(\frac{A}{|i|+i_0}\right),\] possibly up to a
nonzero multiplicative constant.
After some straightforward manipulations, one quickly finds that the hypothesis \[\sup_{x> 0}\frac{\alpha(x)}{x\alpha'(x)}<\infty\] is needed to
make the linear estimate work, and this is precisely one of the hypothesis of Theorem~\ref{t:int-comparison}. This
last supremum is easily seen
not to be finite for arbitrary concave moduli and therefore its finiteness imposes a nontrivial hypothesis on $\alpha$.
As an example, one can consider
$\alpha(x)=1/\log{(1/x)}$.
\end{rem}

\subsection{Remarks on McDuff's Question}
Let $f$ be an exceptional $C^1$--diffeomorphism, and let $\ell_i=|f^i(J)|$ for a wandering interval $J\sse S^1$. 
It is evident from the discussion in this section that the assumption $\ell_{i+1}/\ell_i\to 1$ for the lengths of the successive wandering intervals
implies the strong conclusion that $f'=1$ on the endpoints of the wandering intervals. This is a somewhat restrictive phenomenon, and the
methods of this paper are not suited to address the possibility that $f'$ is not identically $1$ on the endpoints. For instance, the fundamental
estimate, the $1$ appearing in the expression \[\left(1-\frac{\ell_{i+1}}{\ell_i}\right)\] must be replaced by the corresponding endpoint derivative.
This is destructive in certain inductive procedures involved in the proof of Theorem~\ref{t:alpha-inv}, since certain products which one
needs to be bounded away from zero become products of derivatives of $f$ at endpoints of a maximal wandering set for $f$.

There is a long-standing open question about exceptional diffeomorphisms of the circle due to D. McDuff~\cite{McDuff1981,athanassopoulos}.
In her work, she considers the
\emph{ratio spectrum} for a maximal wandering set of an exceptional diffeomorphism. Namely, she rearranges the lengths of the intervals
to form a decreasing sequence $\{\lambda_i\}_{i\in\N}$, and considers the possible accumulation points of the set $\lambda_i/\lambda_{i+1}$.
She proves that this set is bounded and that $1$ is an accumulation point, and asks if $1$ is in fact the only accumulation point. One
can naturally strengthen her question as follows:

\begin{que}\label{q:mcduff-strong}
Let $U\subset S^1$ be an open interval meeting the exceptional minimal set of a finitely generated group, and let $\{\lambda_i\}_{i\in \N}$ be the lengths of intervals in a
maximal wandering set lying in $U$, arranged in decreasing order. Does it follow that $\lambda_i/\lambda_{i+1}\to 1$?
\end{que}

Even assuming a positive answer to this strengthened version of McDuff's question, and even after possibly assuming strong Diophantine
properties for the rotation number $\theta$,
it seems impossible to remove the assumptions
$\ell_{i+1}/\ell_i\to 1$ in our results. The reason for this is the mismatch which occurs upon rearranging the lengths of intervals. One may control
the first return time of a wandering interval to $U$ via Diophantine properties, but the index rearrangement map
\[i\mapsto\ell_i\mapsto \lambda_j\mapsto j\] may be so badly behaved that one may not be able to relate $\ell_{i+1}/\ell_i$ to
$\lambda_{j+1}/\lambda_j$ in a way which would be useful for our purposes. Lastly, we note the claim in the proof of Theorem~\ref{t:int-comparison} implies an affirmative answer to McDuff's question for the construction there with parameters given by $\ell_y := \alpha^{-1} (c / (\|y\|_0+k))$.

\appendix

\section{A proof of  Theorem~\ref{t:int-alpha}}\label{s:appendix}
In this section we give a proof of Theorem~\ref{t:int-alpha}, which the reader may  use as a guide to understanding the proof of
Theorem~\ref{t:c1al}.

\subsection{Herman's construction of exceptional diffeomorphisms}
In this subsection, our discussion is
more explicitly modeled on the construction of exceptional diffeomorphisms due to Herman~\cite{Herman1979}.
We will suppress the differential when writing integrals when there is no danger of confusion.

\begin{prop}\label{p:g-denjoy}
Let $\alpha$ be a concave modulus, and let $\{(x_i,y_i)\}_{i\in\bZ}$ be a disjoint collection of intervals in $S^1$.
Suppose that there exists a positive, $\alpha$--continuous map $g$ on $S^1$ satisfying the following for all $i\in\bZ$:
\begin{itemize}
\item $\displaystyle \int_{S^1} g=1$;
\item $\displaystyle \int_{x_i}^{y_i} g = |y_{i+1}-x_{i+1}|$;
\item $\displaystyle\int_{x_{i-1}}^{x_i} g =|x_{i+1}-x_{i}|$.
\end{itemize}
Then the map 
\[
f(x):=x_1+\int_{x_0}^x g\]
is a $C^{1,\alpha}$ diffeomorphism of $S^1$ such that $f((x_i,y_i))=(x_{i+1},y_{i+1})$ for all $i$.
\end{prop}

The reader  may  compare this  proposition with Lemma~\ref{l:Yoccoz}  above.
\bp[Proof of Proposition~\ref{p:g-denjoy}]
It is obvious that $f$ is  bijective and $C^{1,\alpha}$.
After dividing the cases that 
\[
x_{i-1}\le x\le x_i\]
 and that \[x_{i-1}\le x_i\le x,\]  one can easily verify the equality
\[
x_i+\int_{x_{i-1}}^x g =x_{i+1}+\int_{x_{i}}^x g.\]
The proof is then immediate.
\ep

It is not very difficult to construct a $g$ satisfying the hypotheses of
Proposition~\ref{p:g-denjoy} with the extra condition $g(x) = 1$ on $S^1\setminus \coprod_i [x_i,y_i]$.
The idea is that we set $g$ to be constantly 1 except for some suitable bumps (up or down) inside each interval $[x_i,y_i]$.
We remark that the method of this subsection, however, does not extend to the case when $\alpha(x) = x\log (1/x)$;
cf.~\cite[Exercise 4.1.26]{Navas2011} and~\cite[Chapter 12]{KH1995}.

For disjoint compact intervals $A=[a,a'],B=[b,b'],C=[c,c']$ of $S^1$, 
we write $A<B<C$ if we have the relation
\[a\le a'<b\le b'<c\le c'<a,\]
where here the relation $<$ is interpreted in the circular order on $S^1$.
\bd
Let $\{J_i\}_{i\in\bZ}$ be a sequence of compact intervals in $S^1$.
Suppose we have the relation
\[J_i < J_j < J_k\]
if and only if we have
\[J_{i+1} <J_{j+1}<J_{k+1}.\]
Then we say that the sequence $\{J_i\}$ is \emph{circular order preserving}.
\ed
We use the notation
\[
\form{x}=x-\floor{x}\in[0,1)\]
for $x\in\bR$. 
\begin{exmp}\label{ex:irr}
Let $\theta$ be a given irrational number, and let $\{\ell_i\}_{i\in\bZ}$ be a positive sequence such that 
\[L:=\sum_i \ell_i \in(0,1].\]
Using the Dirac measure $\delta_p$ for $p\in S^1$, we define a measure $\mu$ on $S^1$ as
\[
\mu := (1-L)\lambda + \sum_i \ell_i \delta_{ \form{i\theta}}.\]
We let $x_i :=\mu[0,\form{i\theta})$ for $i\in \bZ$.
Then
the set
\[\{ J_i:=[x_i,x_i+\ell_i]\}_{i\in\bZ}\]
 is a disjoint, circular order preserving collection of compact intervals in $S^1$. We call $\{J_i\}$ a \emph{blow-up} of the sequence $\{\form{i\theta}\}\sse S^1$.
 \end{exmp}

\begin{prop}\label{p:g-denjoy-prime-1}
Let $\alpha$ be a concave modulus.
Suppose we have a disjoint, circular order preserving sequence $\{J_i=[x_i,y_i]\}_{i\in\bZ}$ of compact intervals in $S^1$ such that 
\[
\lambda\left( [x_{i-1},x_{i}]\setminus \coprod_{k\in\bZ}J_k\right)
=\lambda\left( [x_i,x_{i+1}]\setminus \coprod_{k\in\bZ}J_k\right)\]
for all $i$. 
For $\ell_i:=|J_i|$, we also assume that 
\[
\sup_i \frac1{\alpha(\ell_i)}\left(1-\frac{\ell_{i+1}}{\ell_i}\right)<\infty.\]
Then there exists a $C^{1,\alpha}$ diffeomorphism $f$ of $S^1$ satisfying $f(J_i)=J_{i+1}$ for all $i$.
\end{prop}

The reader will note the appearance of the
conclusion of the fundamental estimate (Lemma~\ref{lem:fundamental}) occurring in the statement of the proposition.

\bp[Proof of Proposition~\ref{p:g-denjoy-prime-1}]
Let $\chi_J$ denote the indicator function of $J\sse S^1$. 
We let $\rho_i$ be an arbitrary smooth function supported on $[0,1]$ such that $\int \rho_i=1$ and such that 
\[
1-(1-\ell_{i+1}/\ell_i)\rho_i(x)>0\]
for all $x$. It is a straightforward exercise to produce such a function $\rho_i$ for each $i$,
and in fact one may assume that $\rho_i$ is bounded
independently of
$i$, since we have that \[\inf_i \frac{\ell_{i+1}}{\ell_i}>0\] by the fundamental estimate (Lemma~\ref{lem:fundamental}).
Define a positive function
\[
g(x):=1 - \sum_i \left(1-\frac{\ell_{i+1}}{\ell_i}\right)\rho_i\left(\frac{x-x_i}{\ell_i}\right).\]
Such functions have been constructed by many authors;
see Section 12.2 of~\cite{KH1995}, and particularly X.3 of~\cite{Herman1979}, for instance.
It is obvious that the hypotheses of Proposition~\ref{p:g-denjoy} are satisfied. 

For distinct points $x,y\in J_i$, we have
\[
\sup_{x,y\in J_i}\frac{|g(x)-g(y)|}{\alpha(x-y)}
= \left(1-\frac{\ell_{i+1}}{\ell_i}\right) 
\sup_{s,t\in [0,1]} \frac{|\rho_i(s)-\rho_i(t)|}{\alpha(\ell_i(s-t))}
\le
 \frac1{\alpha(\ell_i)}\left(1-\frac{\ell_{i+1}}{\ell_i}\right) \|\rho_i\|_{\mathrm{Lip}}.\]
Note that here we used the fact that $x/\alpha(x)$ is monotone increasing. 
It follows that $[g\restriction_{J_i}]_\alpha$ is bounded. Since $g=1$ outside $\coprod_i J_i$, it follows that
$[f']_\alpha=[g]_\alpha$ is bounded.
\ep

The following may be compared with Corollary~\ref{c:stab0} above.

\begin{cor}\label{c:g-denjoy-prime-1}
Let $\alpha$ be a concave modulus.
If a positive sequence $\{\ell_i\}_{i\in\bZ}$ satisfies that $\sum_i\ell_i\le 1$
and that
\[\sup_i \frac1{\alpha(\ell_i)}\left(1-\frac{\ell_{i+1}}{\ell_i}\right)<\infty,\]
then there exists an exceptional $C^{1,\alpha}$ diffeomorphism $f$ of $S^1$ with a wandering interval $J\sse S^1$,
and such that $|f^i(J)|=\ell_i$
 for all $i$.
\end{cor}
\bp
The corollary is a simple consequence of Example~\ref{ex:irr} and Proposition~\ref{p:g-denjoy-prime-1}.
\ep

\subsection{The proof of the Theorem}
We can now  combine the observations of the previous subsection to provide a  mostly  self-contained proof:
\begin{proof}[Proof of Theorem~\ref{t:int-alpha}]
As follows from the work of Medvedev~\cite{Medvedev2001}, we lose no generality with the assumption that $\alpha$ is smooth.
We put
$
K:=\max (2,1/\alpha(1))$, and set
\[
v(x):= x^2\alpha(1/x).\]
For all $t\ge1$, we have
\[v(x/t) =(x^2/t^2) \cdot \alpha(t/x) \ge (x^2/t^2) \cdot \alpha(1/x) = v(x)/t^2.\]
Since $x/\alpha$ is monotone increasing, 
we have
\[
(v(x)/x)' = ( \alpha(1/x)/ (1/x) )' \ge0.\]
In particular, whenever $x\ge K$ we have
\[v(x)\ge x\cdot v(1)/1 \ge x/K.\]

We also note from the estimate
\[
0\ge (\alpha(x)/x)' = (x\alpha(x)'-\alpha(x))/x^2,\]
we have that $0<x\alpha'(x)\le \alpha(x)$ for all $x$. So, we get
\[
v(x)= x^2\alpha(1/x) \ge 
 x\alpha'(1/x) >0.\]
For all $x\ge K$, we obtain that
 \[xv'(x) = x( 2x\alpha(1/x)-\alpha'(1/x))=2v(x) - x\alpha'(1/x) \in [v(x),2v(x)].\]

We now set \[ \ell_i:= 1/ v(|i|+K)\] for all $i\in\bZ$.
Since \[\int_K^\infty 1/v = \int_0^{1/K} 1/\alpha<\infty,\]
we see that  $\sum_i \ell_i\le 1$, possibly increasing $K$ if necessary.

Let $i\in\bZ$ and set $j=|i|$. 
Note that 
\[
v(j+K)\ge (j+K)/K,\]
and that
\[ v(j+K\pm1) \ge v((j+K)/2)\ge v(j+K)/4.\]
For some $y_0$ between $j+K$ and $j+K\pm1$, we have
\begin{align*}
\frac{1}{\alpha(\ell_j)}
\abs*{1-\frac{\ell_{j+1}}{\ell_j}}
&=\frac{1}{\ell_j^2 v(1/\ell_j)}
\abs*{1-\frac{v(j+K)}{v(j+K\pm1)}}\\
&=
\frac{v(j+K)^2}{v\circ v(j+K)}\cdot
\frac{v'(y_0)}{v(j+K\pm1)} =
\frac{v(j+K)}{v\circ v(j+K)}
\cdot \frac{v(j+K)}{v(j+K\pm1)}\cdot
v'(y_0)\\
&\le
\frac{j/K+1}{v(j/K+1)}\cdot 4 \cdot \frac{2 v(y_0)}{y_0}\le
\frac{j/K+1}{v(2j+2K)/(2K)^2}\cdot \frac{8 v(y_0)}{y_0}\\
&=
32K \frac{j+K}{y_0}\cdot \frac{ v(y_0)}{v(2j+2K)}\le 64K.
\end{align*}

It follows that the quantities \[\frac{1}{\alpha(\ell_j)}
\abs*{1-\frac{\ell_{j+1}}{\ell_j}}\] are uniformly  bounded independently of $j$, which the reader may compare with the claim in the
proof of Theorem~\ref{t:c1al}.
Note that the conditions of Corollary~\ref{c:g-denjoy-prime-1} are now met.

By choosing $L=\sum_i\ell_i\approx0$ in the proof of Proposition~\ref{p:g-denjoy-prime-1}, we may require that 
\[
\|f\|_{1}=\max( \|f-T(\theta)\|, \|f'-1\|)\]
is as small as desired. To see that $\|f-T(\theta)\|\to 0$ as $L\to 0$, we may simply choose the index shift $K$ to be very large, whence
$f$ will converge to rotation by $\theta$. To see that $ \|f'-1\|\to 0$, we note that $ \|f'-1\|$ vanishes if $x\notin J_i$ for some $i$, and
 is equal to \[\left(1-\frac{\ell_{i+1}}{\ell_i}\right)\rho_i\left(\frac{x-x_i}{\ell_i}\right)\] for $x\in J_i$. Since $\rho_i$ is bounded independently of
 $i$ and since $1-\ell_{i+1}/\ell_i\to 0$ as $i\to\infty$ by the fundamental estimate, we see that $\|f'-1\|$ indeed tends to $1$.
\end{proof}

\section*{Acknowledgements}

The authors are grateful to B. Hayes, N. Kang, Y. Lodha, and M. Triestino for helpful comments on an earlier draft of this paper.
The authors thank an  anonymous referee for comments which improved the exposition in  the paper.
The first author is supported by a KIAS Individual Grant (MG073601) at Korea Institute for Advanced Study and by the
National Research Foundation (2018R1A2B6004003).
The second author is partially supported 
by an Alfred P. Sloan Foundation Research Fellowship, and by NSF Grant DMS-1711488. The second author is grateful to the Korea
Institute for Advanced Study for its hospitality while this research was completed.


\providecommand{\bysame}{\leavevmode\hbox to3em{\hrulefill}\thinspace}
\providecommand{\MR}{\relax\ifhmode\unskip\space\fi MR }
\providecommand{\MRhref}[2]{%
  \href{http://www.ams.org/mathscinet-getitem?mr=#1}{#2}
}
\providecommand{\href}[2]{#2}

\end{document}